\documentclass[a4paper, 10pt]{amsart}

\usepackage{amssymb,amsthm,amstext,amsmath,amscd,latexsym,amsfonts,amscd}

\newtheorem{theorem}{Theorem}[section]
\newtheorem{lemma}[theorem]{Lemma}
\newtheorem{proposition}[theorem]{Proposition}
\newtheorem{corollary}[theorem]{Corollary}

\theoremstyle{definition}
\newtheorem{definition}[theorem]{Definition}
\newtheorem{example}[theorem]{Example}

\newtheorem*{question}{Question}

\theoremstyle{remark}
\newtheorem{remark}[theorem]{Remark}

\numberwithin{equation}{section}

\title[Subcategories of extension modules by Serre subcategories]{Subcategories of extension modules by Serre subcategories}

\author{Takeshi Yoshizawa}
\address{Graduate School of Natural Science and Technology, Okayama University, Okayama 700-8530, Japan}
\email{tyoshiza@math.okayama-u.ac.jp}


\begin{document}
%%%%%%%%%%%%%%%%%%%%%%%%%%%%%%%%%%%%%%%%%%%%%%%%%%%%%%%%%%
%%%%%%%%%%%%%%%%%%%%%%%%%%%%%%%%%%%%%%%%%%%%%%%%%%%%%%%%%%
%%%%%%%%%%%%%%%%%%%%%%%%%%%%%%%%%%%%%%%%%%%%%%%%%%%%%%%%%%
%%%%%%%%%%%%%%%%%%%%%%%%%%%%%%%%%%%%%%%%%%%%%%%%%%%%%%%%%%
\begin{abstract}
In 1962, P.\ Gabriel showed the existence of a lattice isomorphism  
between the set of Serre subcategories of the category consisting of finitely generated modules 
and the set of specialization closed subsets of the set of prime ideals.  
In this paper, 
we consider subcategories consisting of the extensions of modules in two given Serre subcategories 
to find a method of constructing Serre subcategories of the category of modules. 
We shall give a criterion for this subcategory to be a Serre subcategory. 
\end{abstract}
%%%%%%%%%%%%%%%%%%%%%%%%%%%%%%%%%%%%%%%%%%%%%%%%%%%%%%%%%%
%%%%%%%%%%%%%%%%%%%%%%%%%%%%%%%%%%%%%%%%%%%%%%%%%%%%%%%%%%
\maketitle
%%%%%%%%%%%%%%%%%%%%%%%%%%%%%%%%%%%%%%%%%%%%%%%%%%%%%%%%%%
%%%%%%%%%%%%%%%%%%%%%%%%%%%%%%%%%%%%%%%%%%%%%%%%%%%%%%%%%%
\section*{Introduction}
\thispagestyle{empty}
Let $R$ be a commutative noetherian ring. 
We denote by $R\text{-}\mathrm{Mod}$ the category of $R$-modules and by $R\text{-}\mathrm{mod}$ the full subcategory of finitely generated $R$-modules. 

In \cite{G}, 
it was proved that one has a lattice isomorphism between the set of Serre subcategories of $R\text{-}\mathrm{mod}$ and the set of specialization closed subsets of $\mathrm{Spec}(R)$ by P.\ Gabriel.
On the other hand,  
A.\ Neeman showed the existence of an isomorphism of lattices between the set of smashing subcategories of the derived category of $R\text{-}\mathrm{Mod}$ 
and the set of specialization closed subsets of $\mathrm{Spec}(R)$ in \cite{Neeman-1992}. 
Later, 
R.\ Takahashi constructed a module version of Neeman's theorem which induces Gabriel's theorem in \cite{Takahashi-2008}. 

A Serre subcategory of an abelian category is defined to be a full subcategory which is closed under subobjects, quotient objects and extensions. 
Recently, 
many authors study the notion of Serre subcategory not only in the category theory but also in the local cohomology theory. (For example see \cite{AM-2008}.)
By above Gabriel's result, we can give all Serre subcategories of $R\text{-}\mathrm{mod}$ by considering specialization closed subsets of $\mathrm{Spec}(R)$. 
However, 
we want to know other way of constructing Serre subcategories of $R\text{-}\mathrm{Mod}$ with a view of treating it in local cohomology theory. 
Therefore, one of the main purposes in this paper is to give this way by considering subcategories consisting of the extensions of modules in two given Serre subcategories.

%%%
\vspace{5pt}
To be more precise, let $\mathcal{S}_{1}$ and $\mathcal{S}_{2}$ be Serre subcategories of $R\text{-}\mathrm{Mod}$. 
We consider a subcategory consisting of extension modules of $\mathcal{S}_{1}$ by $\mathcal{S}_{2}$, that is

\[ 
(\mathcal{S}_{1}, \mathcal{S}_{2}) = \left\{ M \in R\text{-}\mathrm{Mod} \text{ \Large $\mid$ }  
\begin{matrix} \text{there are $S_{1} \in \mathcal{S}_{1}$ and $S_{2} \in \mathcal{S}_{2}$ such that} \cr 
\minCDarrowwidth1pc \begin{CD}0 @>>> S_{1} @>>> M @>>> S_{2} @>>> 0\end{CD} \text{~is exact.}\cr \end{matrix} \right\}. 
\]

\vspace{3pt}
\noindent
For example, modules of a subcategory $(\mathcal{S}_{f. g. }, \mathcal{S}_{Artin })$ are known as Minimax modules 
where $\mathcal{S}_{f. g. }$ is the Serre subcategory consisting of finitely generated $R$-modules and $\mathcal{S}_{Artin }$ is the Serre subcategory consisting of Artinian $R$-modules. 
In \cite{BN-2008}, 
K.\ Bahmanpour and R.\ Naghipour showed that the subcategory $(\mathcal{S}_{f. g. }, \mathcal{S}_{Artin })$ is a Serre subcategory. 
However, a subcategory $(\mathcal{S}_{1}, \mathcal{S}_{2})$ needs not be Serre. 
In fact, a subcategory $(\mathcal{S}_{Artin }, \mathcal{S}_{f. g. })$ is not a Serre subcategory. 
Therefore, we shall give a necessary and sufficient condition for a subcategory $(\mathcal{S}_{1}, \mathcal{S}_{2})$ to be a Serre subcategory  
and several examples of Serre subcategory $(\mathcal{S}_{1}, \mathcal{S}_{2})$. 
In particular, we shall see that the following subcategories are Serre subcategories: 
\begin{enumerate}

\item\, A subcategory $(\mathcal{S}_{1}, \mathcal{S}_{2})$ for Serre subcategories $\mathcal{S}_{1}$ and $\mathcal{S}_{2}$ of $R\text{-}\mathrm{mod}$; 

\vspace{3pt}
\item\, A subcategory $(\mathcal{S}_{f. g. }, \mathcal{S})$ for a Serre subcategory $\mathcal{S}$ of $R\text{-}\mathrm{Mod}$;  

\vspace{3pt}
\item\, A subcategory $(\mathcal{S}, \mathcal{S}_{Artin })$ for a Serre subcategory $\mathcal{S}$ of $R\text{-}\mathrm{Mod}$. 

\end{enumerate}

%%%
\vspace{7pt}
The organization of this paper is as follows.
 
In section 1, 
we shall define the notion of subcategories $(\mathcal{S}_{1}, \mathcal{S}_{2})$ of extension modules related to Serre subcategories $\mathcal{S}_{1}$ and $\mathcal{S}_{2}$ of $R\text{-}\mathrm{Mod}$ and study basic properties.  

In section 2, we shall treat Serre subcategories of $R\text{-}\mathrm{mod}$. 
In particular, we show the above example (1) is a Serre subcategory of $R\text{-}\mathrm{mod}$. (Theorem \ref{serre-Rmod}.) 

The section 3 is a main part of this paper. 
We give a  necessary and sufficient condition for a subcategory $(\mathcal{S}_{1}, \mathcal{S}_{2})$ to be a Serre subcategory. (Theorem \ref{serre}.) 

In section 4, 
we apply our argument to the local cohomology theory involving the condition $(C_{I})$ for an ideal $I$ of $R$ which is defined by M.\ Aghapournahr and  L.\ Melkersson in \cite{AM-2008}.
%%%%%%%%%%%%%%%%%%%%%%%%%%%%%%%%%%%%%%%%%%%%%%%%%%%%%%%%%
%%%%%%%%%%%%%%%%%%%%%%%%%%%%%%%%%%%%%%%%%%%%%%%%%%%%%%%%%%

%%%%%%%%%%%%%%%%%%%%%%%
%%%%%%%%%%%%%%%%%%%%%%%
%       Section1
%%%%%%%%%%%%%%%%%%%%%%%
%%%%%%%%%%%%%%%%%%%%%%%
\vspace{7pt}
\section{The definition of a subcategory of extension modules by Serre subcategories}
Throughout this paper, all rings are commutative noetherian and all modules are unitary. 
We denote by $R\text{-}\mathrm{Mod}$ the category of $R$-modules and by $R\text{-}\mathrm{mod}$ the full subcategory of finitely generated $R$-modules. 
We assume that all full subcategories $\mathcal{X}$ of an abelian category $\mathcal{A}$ are closed under isomorphisms, 
that is if $X$ is an object of $\mathcal{X}$ and $Y$ is an object of $\mathcal{A}$ which is isomorphic to $X$, then $Y$ is also an object of $\mathcal{X}$. 
We recall that a full subcategory of an abelian category is said to be Serre subcategory if it is closed under subobjects, quotient objects and extensions.

\vspace{5pt}
In this section, 
we shall give the definition of a subcategory of extension modules by Serre subcategories 
and study basic properties. 

%\vspace{1pt}
\begin{definition}
Let $\mathcal{S}_1$ and $\mathcal{S}_2$ be Serre subcategories of $R\text{-}\mathrm{Mod}$. 
We denote by $(\mathcal{S}_1, \mathcal{S}_2)$ a subcategory consisting of $R$-modules $M$ with a short exact sequence  
\[ 0 \to S_{1} \to M \to S_{2} \to 0\]
of $R$-modules where each $S_{i}$ is in $\mathcal{S}_{i}$, that is
\[ 
(\mathcal{S}_{1}, \mathcal{S}_{2}) = \left\{ M \in R\text{-}\mathrm{Mod} \text{ \Large $\mid$ }  
\begin{matrix} \text{there are $S_{1} \in \mathcal{S}_{1}$ and $S_{2} \in \mathcal{S}_{2}$ such that} \cr 
\minCDarrowwidth1pc \begin{CD}0 @>>> S_{1} @>>> M @>>> S_{2} @>>> 0\end{CD} \text{~is exact.}\cr \end{matrix} \right\}. 
\]
We shall refer to a subcategory $(\mathcal{S}_{1}, \mathcal{S}_{2})$ as a subcategory of extension modules of $\mathcal{S}_{1}$ by $\mathcal{S}_{2}$. 
\end{definition}

%\vspace{1pt}
\begin{remark}\label{basic}
Let $\mathcal{S}_1$ and $\mathcal{S}_2$ be Serre subcategories of $R\text{-}\mathrm{Mod}$. 

\vspace{3pt}
\noindent
(1)\, Since the zero module belongs to any Serre subcategory, 
one has $\mathcal{S}_{1} \subseteq (\mathcal{S}_{1}, \mathcal{S}_{2})$ and $\mathcal{S}_{2}\subseteq (\mathcal{S}_{1}, \mathcal{S}_{2})$. 

\vspace{3pt}
\noindent
(2)\, It holds $\mathcal{S}_{1} \supseteq \mathcal{S}_{2}$ if and only if $(\mathcal{S}_{1}, \mathcal{S}_{2})=\mathcal{S}_{1}$.

\vspace{3pt}
\noindent
(3)\, It holds $\mathcal{S}_{1} \subseteq \mathcal{S}_{2}$ if and only if  $(\mathcal{S}_{1}, \mathcal{S}_{2})=\mathcal{S}_{2}$.

\vspace{3pt}
\noindent
(4)\, A subcategory $(\mathcal{S}_{1},\mathcal{S}_{2})$ is closed under finite direct sums. 
\end{remark}

%\vspace{1pt}
\begin{example}
We denote by $\mathcal{S}_{f. g. }$ the subcategory consisting of finitely generated $R$-modules and by $\mathcal{S}_{Artin }$ the subcategory consisting of Artinian $R$-modules. 
Then a subcategory $(\mathcal{S}_{f. g. }, \mathcal{S}_{Artin })$ is known as the subcategory consisting of Minimax $R$-modules 
and a subcategory $(\mathcal{S}_{Artin }, \mathcal{S}_{f. g. })$ is known as the subcategory consisting of Maxmini $R$-modules.   
\end{example}

\vspace{3pt}
%%%%%%%%%%%%%%%%%%%%%%%%%%%%%%%%%%%%%%%%%%%%%%%%
%
%	closed under submodules and quotient
%
%%%%%%%%%%%%%%%%%%%%%%%%%%%%%%%%%%%%%%%%%%%%%%%%%
\begin{proposition}\label{cusq}
Let $\mathcal{S}_{1}$ and  $\mathcal{S}_{2}$ be Serre subcategories of $R\text{-}\mathrm{Mod}$. 
Then a subcategory $(\mathcal{S}_{1}, \mathcal{S}_{2})$ is closed under submodules and quotient modules.
\end{proposition}

\begin{proof}
Let $0 \to L \to M \to N \to 0$ be a short exact sequence of $R$-modules. 
We assume that $M$ is in $(\mathcal{S}_{1}, \mathcal{S}_{2})$ and shall show that $L$ and $N$ are in $(\mathcal{S}_{1}, \mathcal{S}_{2})$. 

It follows from the definition of $(\mathcal{S}_{1}, \mathcal{S}_{2})$ that 
there exists a short exact sequence  
\[ \minCDarrowwidth1pc \begin{CD}0 @>>> S_{1} @>{\varphi}>> M @>>> S_{2} @>>> 0\end{CD} \]
of $R$-modules where $S_{1}$ is in $\mathcal{S}_{1}$ and $S_{2}$ is in $\mathcal{S}_{2}$.  
Then we can construct the following commutative diagram 

\[ \begin{CD}
@. 0 @. 0 @. 0 @. \\
@. @VVV @VVV @VVV \\
0 @>>> S_{1} \cap L @>>> S_{1} @>>> \displaystyle\frac{S_{1}}{S_{1} \cap L} @>>> 0 \\
@. @VVV @VV{\varphi}V @VV{\overline{\varphi}}V \\
0 @>>> L  @>>> M @>>> N @>>> 0 \\
@. @VVV @VVV @VVV \\
0 @>>> \displaystyle\frac{L}{S_{1} \cap L}   @>>> S_{2} @>>> N^{\prime} @>>> 0 \\
@. @VVV @VVV @VVV \\
@. 0 @. 0 @. 0 @.\vspace{3pt}
\end{CD} \]

\vspace{3pt}
\noindent
of $R$-modules with exact rows and columns 
where $\overline{\varphi}$ is a natural map induced by $\varphi$ and $N^{\prime}=\mathrm{Coker} (\overline{\varphi})$. 
Since each subcategory $\mathcal{S}_{i}$ is closed under submodules and quotient modules, 
we can see that $S_{1} \cap L$, $S_{1}/(S_{1}\cap L)$ are in $\mathcal{S}_{1}$ and $L/(S_{1} \cap L)$, $N^{\prime}$ are in $\mathcal{S}_{2}$.  
Consequently, $L$ and $N$ are in $(\mathcal{S}_{1}, \mathcal{S}_{2})$.
\end{proof}

\vspace{5pt}
Here, a natural question arises. 

\begin{question}  
Is a subcategory $(\mathcal{S}_{1}, \mathcal{S}_{2})$ Serre subcategory for Serre subcategories $\mathcal{S}_{1}$ and $\mathcal{S}_{2}$?
\end{question}

\vspace{3pt}
\noindent
For example, K.\ Bahmanpour and R.\ Naghipour showed that the subcategory $(\mathcal{S}_{f. g. }, \mathcal{S}_{Artin })$ is a Serre subcategory in \cite[Lemma 2.1]{BN-2008}. 
Proposition \ref{cusq} states that a subcategory $(\mathcal{S}_{1}, \mathcal{S}_{2})$ is a Serre subcategory if this subcategory is closed under extensions.  
However, the conclusion in the above question does not hold in general.

\begin{example}\label{Example-nonSerre}
We shall see the subcategory $(\mathcal{S}_{Artin }, \mathcal{S}_{f. g. })$ needs not be closed under extensions. 

Let $(R, \mathfrak{m})$ be a Gorenstein local ring of dimension one with maximal ideal $\mathfrak{m}$.
Then one has a minimal injective resolution 
\[ 0 \to R \to \bigoplus_{ \text{$\begin{matrix}\mathfrak{p} \in \mathrm{Spec}(R),\cr \mathrm{ht} \mathfrak{p}=0 \end{matrix}$}} E_{R}(R/\mathfrak{p} ) \to E_{R}(R/\mathfrak{m}) \to 0 \] of $R$. 
($E_{R}(M)$ denotes the injective hull of an $R$-module $M$.)
We note that $R$ and $E_{R}(R/\mathfrak{m})$ are in $(\mathcal{S}_{Artin }, \mathcal{S}_{f. g. })$. 

Now, we assume that $(\mathcal{S}_{Artin }, \mathcal{S}_{f. g. })$ is closed under extensions. 
Then $E_{R}(R)=\oplus_{\mathrm{ht} \mathfrak{p}=0} E_{R}(R/\mathfrak{p})$ is in $(\mathcal{S}_{Artin }, \mathcal{S}_{f. g. })$. 
It follows from the definition of $(\mathcal{S}_{Artin }, \mathcal{S}_{f. g. })$ that there exists an Artinian $R$-submodule $N$ of $E_{R}(R)$ such that $E_{R}(R)/N$ is a finitely generated $R$-module. 

If $N=0$, then $E_{R}(R)$ is a finitely generated injective $R$-module. 
It follows from the Bass formula that one has $\mathrm{dim}\, R=\mathrm{depth}\, R= \mathrm{inj\,dim}\, E_{R}(R)=0$. 
However, this equality contradicts $\mathrm{\dim}\, R=1$. 
On the other hand, if $N\neq 0$, then $N$ is a non-zero Artinian $R$-module. 
Hence, one has $\mathrm{Ass}_{R}(N)=\{ \mathfrak{m} \}$. 
Since $N$ is an $R$-submodule of $E_{R}(R)$, one has 
\[ \mathrm{Ass}_{R}(N) \subseteq \mathrm{Ass}_{R}(E_{R}(R))=\{ \mathfrak{p} \in \mathrm{Spec}(R) \mid \mathrm{ht}\, \mathfrak{p}=0\}. \]
This is contradiction as well.
\end{example}

\vspace{7pt}
%%%%%%%%%%%%%%%%%%%%%%%%%%%%%%%%%%%%%%%%%%%%
%
% Serre subcategory of $R-mod$
%
%
%%%%%%%%%%%%%%%%%%%%%%%%%%%%%%%%%%%%%%%%%%%%
\section{Subcategories of extension modules by Serre subcategories of $R\text{-}\mathrm{mod}$}

%%%%%%%%%%%%%%%%%%%%%%%%%%%%%%%%%%%
In this section, 
we shall see that a subcategory $(\mathcal{S}_{1}, \mathcal{S}_{2})$ of $R\text{-}\mathrm{mod}$ is Serre for Serre subcategories $\mathcal{S}_{1}$ and $\mathcal{S}_{2}$ of $R\text{-}\mathrm{mod}$. 

\vspace{5pt} 
Recall that a subset $W$ of $\mathrm{Spec}(R)$ is said to be a specialization closed subset, 
provided that if $\mathfrak{p}$ is a prime ideal in $W$ and $\mathfrak{q}$ is a prime ideal containing $\mathfrak{p}$ then $\mathfrak{q}$ is also in $W$. 
We denote by $\varGamma_{W}$ the section functor with support in a specialization closed subset $W$ of $\mathrm{Spec}(R)$, that is 
\[ \varGamma_{W}(M)=\{x \in M \mid \mathrm{Supp}_{R}(Rx) \subseteq W \} \]
for each $R$-module $M$.  

\vspace{5pt}
Let us start to prove the following  two lemmas.

%%%%%%%%%%%%%%%%%%%%%%
%
%   disjoint lemma
%
%%%%%%%%%%%%%%%%%%%%%%
\begin{lemma}\label{DL}
Let $M$ be an $R$-module and $W$ be a specialization closed subset of $\mathrm{Spec}(R)$. 
Then $\mathrm{Ass}_{R}(\varGamma_{W}(M))$ and $\mathrm{Ass}_{R}(M/\varGamma_{W}(M))$ are disjoint, and that
\[ \mathrm{Ass}_{R}(M)=\mathrm{Ass}_{R}(\varGamma_{W}(M))  \cup \mathrm{Ass}_{R}(M/\varGamma_{W}(M)).\] 
\end{lemma}

\begin{proof}
It is clear that one has $\mathrm{Ass}_{R}(\varGamma_{W}(M)) \subseteq W$ and $\mathrm{Ass}_{R}(M/\varGamma_{W}(M)) \cap W = \emptyset$. 
Hence, $\mathrm{Ass}_{R}(\varGamma_{W}(M)) \cup \mathrm{Ass}_{R}(M/\varGamma_{W}(M))$ is disjoint union. 

To see that the equality in our assertion holds, it is enough to show that one has $\mathrm{Ass}_{R}(M/\varGamma_{W}(M)) \subseteq \mathrm{Ass}_{R}(M)$. 
Let $E_{R}(M)$ be an injective hull of $M$. 
Then $\varGamma_{W}(E_{R}(M))$ is also an injective $R$-module. (Also see \cite[Theorem 2.6]{YY}.)
We consider the following commutative diagram

\[ \minCDarrowwidth1.7pc
\begin{CD}
0 @>>> \varGamma_{W}(M)@>>> M @>>> M/\varGamma_{W}(M) @>>> 0 \\
@. @VV{\varGamma_{W}(\varphi)=\varphi |_{\varGamma_{W}(M)}}V @VV{\varphi}V @VV{\overline{\varphi}}V \\
0 @>>> \varGamma_{W}(E_{R}(M))  @>>> E_{R}(M) @>>> E_{R}(M)/\varGamma_{W}(E_{R}(M)) @>>> 0 \\
\end{CD} \]

\vspace{3pt}
\noindent
of $R$-modules with exact rows 
where $\varphi$ is an inclusion map from $M$ to $E_{R}(M)$ and $\overline{\varphi}$ is a homomorphism induced by $\varphi$.  
By the injectivity of $\varGamma_{W}(E_{R}(M))$, we can see that the second exact sequence in the above diagram splits, 
so that one has
\[ E_{R}(M)=\varGamma_{W}(E_{R}(M))\oplus E_{R}(M)/\varGamma_{W}(E_{R}(M)).\] 
%and $E_{R}(M)/\varGamma_{W}(E_{R}(M))$ is an injective $R$-module. 

Here, we note that $\overline{\varphi}$ is a monomorphism. 
Actually, we assume $m \in M$ such that $\overline{\varphi}( m+\varGamma_{W}(M))=0 \in E_{R}(M)/\varGamma_{W}(E_{R}(M))$.   
Then it holds $m \in M \cap \varGamma_{W}(E_{R}(M))=\varGamma_{W}(M)$. 
Thus we have $m+\varGamma_{W}(M)=0 \in M/\varGamma_{W}(M)$. 
Consequently, we can see that one has   
\begin{align*}
\mathrm{Ass}_{R} (M/\varGamma_{W}(M)) 
&\subseteq \mathrm{Ass}_{R} \bigl( E_{R}(M)/\varGamma_{W}(E_{R}(M)) \bigr)\\ 
&\subseteq \mathrm{Ass}_{R} (E_{R}(M))\\
&= \mathrm{Ass}_{R} (M).
\end{align*}

The proof is completed.
\end{proof}

%%%%%%%%%%%%%%%%%%%%%%%%%%%
%
%  lemma of Serre of R-mod
%
%%%%%%%%%%%%%%%%%%%%%%%%%%%%
If Serre subcategories $\mathcal{S}_{1}$ and $\mathcal{S}_{2}$ are related to specialization closed subsets of $\mathrm{Spec}(R)$, 
then the structure of a subcategory $(\mathcal{S}_{1}, \mathcal{S}_{2})$ is simple. 
 
\begin{lemma}\label{lem-SerreofRmod}
Let $W_{1}$ and $W_{2}$  be specialization closed subsets of $\mathrm{Spec}(R)$. 
Then the following assertions hold.
\begin{enumerate}
\item \,  
We set $\mathcal{S}_{i} =\{ M \in R\text{-}\mathrm{Mod} \mid \mathrm{Supp}_{R}(M) \subseteq W_{i} \}$  for $i=1$, $2$.
Then one has  
\[ (\mathcal{S}_{1}, \mathcal{S}_{2}) =\{ M \in R\text{-}\mathrm{Mod} \mid \mathrm{Supp}_{R}(M) \subseteq W_{1} \cup W_{2} \}. \] 

\vspace{3pt}
\item \, 
We set $\mathcal{S}_{i} =\{ M \in R\text{-}\mathrm{mod} \mid \mathrm{Supp}_{R}(M) \subseteq W_{i} \}$  for $i=1$, $2$.
Then one has 
\[ (\mathcal{S}_{1}, \mathcal{S}_{2}) =\{ M \in R\text{-}\mathrm{mod} \mid \mathrm{Supp}_{R}(M) \subseteq W_{1} \cup W_{2} \}. \]
\end{enumerate}
\noindent
In particular, subcategories $(\mathcal{S}_{1}, \mathcal{S}_{2})$ in $(1)$ and $(2)$ are Serre subcategories. 
\end{lemma}

\begin{proof} 
\noindent
(1)\, 
If $M$ is in $(\mathcal{S}_{1}, \mathcal{S}_{2})$, 
then there exists a short exact sequence \[0 \to S_{1} \to M \to S_{2} \to 0\] of $R$-modules such that $S_{1}$ is in  $\mathcal{S}_{1}$ and $S_{2}$ is in $\mathcal{S}_{2}$. 
Then $\mathrm{Supp}_{R}(M)=\mathrm{Supp}_{R}(S_{1}) \cup \mathrm{Supp}_{R}(S_{2}) \subseteq W_{1} \cup W_{2}$. 

Conversely, let $M$ be an $R$-module with $\mathrm{Supp}_{R}(M) \subseteq W_{1} \cup W_{2}$. 
We consider a short exact sequence \[ 0 \to \varGamma_{W_{1}}(M) \to M \to M/\varGamma_{W_{1}}(M) \to 0.\]
We note that one has $\mathrm{Supp}_{R}(\varGamma_{W_{1}}(M)) \subseteq W_{1}$. 
Therefore, to prove our assertion, 
it is enough to show that one has $\mathrm{Ass}_{R}(M/\varGamma_{W_{1}}(M)) \subseteq W_{2}$. 
It follows from Lemma \ref{DL} that we have  
\[ \mathrm{Ass}_{R}(M/\varGamma_{W_{1}}(M))\subseteq \mathrm{Ass}_{R}(M) \subseteq \mathrm{Supp}_{R}(M) \subseteq W_{1} \cup W_{2}. \] 
Since $\mathrm{Ass}_{R} (M/\varGamma_{W_{1}}(M)) \cap W_{1} =\emptyset$, one has $\mathrm{Ass}_{R}(M/\varGamma_{W_{1}}(M))\subseteq W_{2}$.  
Consequently, $M$ is in $(\mathcal{S}_{1}, \mathcal{S}_{2})$.   

\noindent
(2)\, We can show the assertion by the same argument in (1).
\end{proof}
%%%%

\vspace{5pt}
Now, we can show the purpose of this section.

\begin{theorem}\label{serre-Rmod}
Let $\mathcal{S}_{1}$ and $\mathcal{S}_{2}$ be Serre subcategories of $R\text{-}\mathrm{mod}$. 
Then a subcategory $(\mathcal{S}_{1}, \mathcal{S}_{2})$ is a Serre subcategory of $R\text{-}\mathrm{mod}$. 
\end{theorem}

\begin{proof}
By Gabriel's result in \cite{G}, 
there is a bijection between the set of Serre subcategories of $R\text{-}\mathrm{mod}$ and the set of specialization closed subsets of $\mathrm{Spec}(R)$. (Also see \cite[Theorem 4.1]{Takahashi-2008}.)
Thus, there exists a specialization closed subset $W_{1}$ (resp. $W_{2}$) of $\mathrm{Spec}(R)$ corresponding to the Serre subcategory $\mathcal{S}_{1}$ (resp. $\mathcal{S}_{2}$). 
In particular, we can denote   
\[\mathcal{S}_{i} =\{  M\in R\text{-}\mathrm{mod} \mid \mathrm{Supp}_{R}(M) \subseteq W_{i} \} \hspace{5pt} \text{and} \hspace{5pt} W_{i} =\bigcup_{M \in \mathcal{S}_{i}} \mathrm{Supp}_{R} (M) \]
for each $i$. 
It follows from Lemma \ref{lem-SerreofRmod} that one has   
\[ (\mathcal{S}_{1}, \mathcal{S}_{2})=\{ M \in R\text{-}\mathrm{mod} \mid \mathrm{Supp}_{R}(M) \subseteq W_{1} \cup W_{2} \} \]
and this subcategory is a Serre subcategory of $R\text{-}\mathrm{mod}$.
\end{proof}

\vspace{7pt}
%%%%%%%%%%%%%%%%%%%%%%%%%%%%%%%%%%%%%%%%%%%%%%%%%%%%%%%%%%%%%%%%%%%%%%%%%%%%%%%%%%%%
%
%	
%
%%%%%%%%%%%%%%%%%%%%%%%%%%%%%%%%%%%%%%%%%%%%%%%%%%%%%%%%%%%%%%%%%%%%%%%%%%%%%%%%%%%%
\section{A criterion for a subcategory $(\mathcal{S}_{1}, \mathcal{S}_{2})$ to be Serre}

Let $\mathcal{S}_{1}$ and $\mathcal{S}_{2}$ be Serre subcategories of $R\text{-}\mathrm{Mod}$.
In this section, 
we shall give a necessary and sufficient condition for a subcategory $(\mathcal{S}_{1}, \mathcal{S}_{2})$ to be a Serre subcategory  
and several examples of Serre subcategory $(\mathcal{S}_{1}, \mathcal{S}_{2})$. 
In particular, we shall show that subcategories $(\mathcal{S}_{f. g. }, \mathcal{S})$ and $(\mathcal{S}, \mathcal{S}_{Artin})$ are Serre subcategories for a Serre subcategory $\mathcal{S}$ of $R\text{-}\mathrm{Mod}$.

\vspace{5pt} 
We start to prove the following lemma. 
If a subcategory $(\mathcal{S}_{1}, \mathcal{S}_{2})$ is a Serre subcategory, 
then we have already seen that one has $(\mathcal{S}_{1},(\mathcal{S}_{1}, \mathcal{S}_{2}))=((\mathcal{S}_{1}, \mathcal{S}_{2}), \mathcal{S}_{2})=(\mathcal{S}_{1},\mathcal{S}_{2})$  in Remark \ref{basic}.  
However, we can see that the following assertion always holds without such an assumption.

\begin{lemma}\label{comp}
Let $\mathcal{S}_{1}$ and $\mathcal{S}_{2}$ be Serre subcategories of $R\text{-}\mathrm{Mod}$. 
We suppose that a sequence $ 0 \to L \to M \to N \to 0$ of $R$-modules is exact. 
Then the following assertions hold. 
\begin{enumerate}
\item\, If $L\in \mathcal{S}_{1}$ and $N \in (\mathcal{S}_{1}, \mathcal{S}_{2})$, then $M \in (\mathcal{S}_{1}, \mathcal{S}_{2})$.
\vspace{3pt}
\item\, If $L \in (\mathcal{S}_{1}, \mathcal{S}_{2})$ and $N \in \mathcal{S}_{2}$, then $M \in (\mathcal{S}_{1}, \mathcal{S}_{2})$.
\end{enumerate}
\end{lemma}

\begin{proof}
\noindent
(1)\, 
We assume that $L$ is in $\mathcal{S}_{1}$ and $N$ is in $(\mathcal{S}_{1}, \mathcal{S}_{2})$. 
Since $N$ belongs to $(\mathcal{S}_{1}, \mathcal{S}_{2})$, 
there exists a short exact sequence 
\[ 0 \to S \to N \to T \to 0\]
of $R$-modules where $S$ is in $\mathcal{S}_{1}$ and $T$ is in $\mathcal{S}_{2}$. 
Then we consider the following pull buck diagram

$$  
\begin{CD}
@.  @.0  @.0  @.  \\
@. @. @VVV @VVV \\
0 @>>> L @>>>S^{\prime}  @>>> S @>>> 0\\
@. \parallel  @.  @VVV @VVV \\
0 @>>> L  @>>> M @>>> N @>>> 0 \\
@. @. @VVV  @VVV  \\
@. @. T @=  T @. \\
@. @. @VVV @VVV \\
@. @. 0 @. 0 @. 
\end{CD} $$

\vspace{3pt}
\noindent
of $R$-modules with exact rows and columns. 
Since $\mathcal{S}_{1}$ is a Serre subcategory, 
it follows from the first row in the diagram that $S^{\prime}$ belongs to $\mathcal{S}_{1}$. 
Consequently, we see that $M$ is in $(\mathcal{S}_{1}, \mathcal{S}_{2})$ by the middle column in the diagram.

\noindent
(2) We assume that $L$ is in $(\mathcal{S}_{1}, \mathcal{S}_{2})$ and $N$ is in $\mathcal{S}_{2}$. 
Since $L$ belongs to $(\mathcal{S}_{1}, \mathcal{S}_{2})$, 
there exists an exact sequence 
\[0 \to S \to L \to T \to 0 \]
of $R$-modules where $S$ is in $\mathcal{S}_{1}$ and $T$ is in $\mathcal{S}_{2}$. 
Then we consider the following push out diagram

$$  
\begin{CD}
@. 0 @. 0 @.  @. \\
@. @VVV @VVV @. \\
@. S @= S @. @. \\
@. @VVV @VVV @. \\
0 @>>> L  @>>> M @>>> N @>>> 0 \\
@. @VVV @VVV @. \hspace{-3.6cm}{\parallel}  \\
0 @>>> T @>>> T^{\prime} @>>>  N @>>>0 \\
@. @VVV @VVV @. \\
@. 0 @. 0 @. @. 
\end{CD} $$

\vspace{3pt}
\noindent
of $R$-modules with exact rows and columns.
Since $\mathcal{S}_{2}$ is a Serre subcategory,  
it follows from the third row in the diagram that $T^{\prime}$ belongs to $\mathcal{S}_{2}$. 
Therefore, we see that $M$ is in $(\mathcal{S}_{1}, \mathcal{S}_{2})$ by the middle column in the diagram.
\end{proof}

%%%%%%%%%%%%%%%%%%%%%%%%%%%%%%%%%%%%%%
%
%   main result
%
%
%%%%%%%%%%%%%%%%%%%%%%%%%%%%%%%%%%%%%%%%%
\vspace{5pt}
Now, we can show the main purpose of this paper. 

\begin{theorem}\label{serre}
Let $\mathcal{S}_{1}$ and $\mathcal{S}_{2}$ be Serre subcategories of $R\text{-}\mathrm{Mod}$. 
Then the following conditions are equivalent: 

\begin{enumerate}
\item\, A subcategory $(\mathcal{S}_{1}, \mathcal{S}_{2})$ is a Serre subcategory; 

\vspace{3pt}
\item\, One has $(\mathcal{S}_{2}, \mathcal{S}_{1}) \subseteq (\mathcal{S}_{1}, \mathcal{S}_{2})$.  
\end{enumerate}
\end{theorem}

\begin{proof}
$(1) \Rightarrow (2)$ 
We assume that $M$ is in $(\mathcal{S}_{2}, \mathcal{S}_{1})$. 
By the definition of a subcategory $(\mathcal{S}_{2}, \mathcal{S}_{1})$, 
there exists a short exact sequence
\[ 0 \to S_{2} \to M \to S_{1} \to 0 \]
of $R$-modules where $S_{1}$ is in $\mathcal{S}_{1}$ and $S_{2}$ is in $\mathcal{S}_{2}$. 
We note that $S_{1}$ is in $\mathcal{S}_{1} \subseteq (\mathcal{S}_{1}, \mathcal{S}_{2})$ and  $S_{2}$ is in $\mathcal{S}_{2} \subseteq (\mathcal{S}_{1}, \mathcal{S}_{2})$. 
Since a subcategory $(\mathcal{S}_{1}, \mathcal{S}_{2})$ is closed under extensions by the assumption (1), 
we see that $M$ is in $(\mathcal{S}_{1}, \mathcal{S}_{2})$. 

\vspace{3pt}
\noindent
$(2) \Rightarrow (1)$ 
We  only have to prove that a subcategory $(\mathcal{S}_{1}, \mathcal{S}_{2})$ is closed under extensions by Proposition \ref{cusq}.
Let $0 \to L \to M \to N \to 0$ be a short exact sequence of $R$-modules such that $L$ and $N$ are in $(\mathcal{S}_{1}, \mathcal{S}_{2})$. 
We shall show that $M$ is also in $(\mathcal{S}_{1}, \mathcal{S}_{2})$. 

Since $L$ is in $(\mathcal{S}_{1}, \mathcal{S}_{2})$, there exists a short exact sequence
\[ 0 \to S \to L \to L/S \to 0 \] 
of $R$-modules where $S$ is in $\mathcal{S}_{1}$ such that $L/S$ is in $\mathcal{S}_{2}$. 
We consider the following push out diagram 

$$  
\begin{CD}
@. 0 @. 0 @.  @. \\
@. @VVV @VVV @. \\
@. S @= S @. @. \\
@. @VVV @VVV @. \\
0 @>>> L  @>>> M @>>> N @>>> 0 \\
@. @VVV @VVV @. \hspace{-3.6cm}{\parallel}  \\
0 @>>> L/S @>>> P @>>>  N @>>>0 \\
@. @VVV @VVV @. \\
@. 0 @. 0 @.  @. 
\end{CD} $$

\vspace{5pt}
\noindent
of $R$-modules with exact rows and columns. 
Next, since $N$ is in $(\mathcal{S}_{1}, \mathcal{S}_{2})$, we have a short exact sequence
\[ 0 \to T \to N \to N/T \to 0 \] 
of $R$-modules where $T$ is in $S_{1}$ such that $N/T$ is in $\mathcal{S}_{2}$. 
We consider the following pull back diagram 

$$  
\begin{CD}
@.  @.0  @.0  @.  \\
@. @. @VVV @VVV \\
0 @>>> L/S @>>>P^{\prime}  @>>> T @>>> 0\\
@. \parallel  @.  @VVV @VVV \\
0 @>>> L/S  @>>> P @>>> N @>>> 0 \\
@. @. @VVV  @VVV  \\
@. @. N/T @=  N/T @. \\
@. @. @VVV @VVV \\
@. @. 0 @. 0 @. 
\end{CD} $$

\vspace{5pt}
\noindent
of $R$-modules with exact rows and columns. 

In the first row of the second diagram, 
since $L/S$ is in $\mathcal{S}_{2}$ and $T$ is in $\mathcal{S}_{1}$, $P^{\prime}$ is in $(\mathcal{S}_{2}, \mathcal{S}_{1})$. 
Now here, it follows from the assumption (2) that $P^{\prime}$ is in $(\mathcal{S}_{1}, \mathcal{S}_{2})$. 
Next, in the middle column of the second diagram, 
we have the short exact sequence such that $P^{\prime}$ is in $(\mathcal{S}_{1}, \mathcal{S}_{2})$ and $N/T$ is in $\mathcal{S}_{2}$.  
Therefore, it follows from Lemma \ref{comp} that $P$ is in $(\mathcal{S}_{1}, \mathcal{S}_{2})$. 
Finally, in the middle column of the first diagram, 
there exists the short exact sequence such that $S$ is in $\mathcal{S}_{1}$ and $P$ is in $(\mathcal{S}_{1}, \mathcal{S}_{2})$.  
Consequently, we see that $M$ is in $(\mathcal{S}_{1}, \mathcal{S}_{2})$ by Lemma \ref{comp}. 

The proof is completed.
\end{proof}

\vspace{5pt}
In the rest of this section, 
we shall give several examples of Serre subcategory $(\mathcal{S}_{1}, \mathcal{S}_{2})$ of $R\text{-}\mathrm{Mod}$. 
The first example is a generalization of \cite[Lemma 2.1]{BN-2008} which states that a subcategory $(\mathcal{S}_{f. g. }, \mathcal{S}_{Artin })$ is a Serre subcategory.

\begin{corollary}\label{FGSerre}
A subcategory $(\mathcal{S}_{f. g. }, \mathcal{S})$ is a Serre subcategory for a Serre subcategory $\mathcal{S}$ of $R\text{-}\mathrm{Mod}$. 
\end{corollary}

\begin{proof}
Let $\mathcal{S}$ be a Serre subcategory of $R\text{-}\mathrm{Mod}$. 
To prove our assertion, it is enough to show that one has $(\mathcal{S}, \mathcal{S}_{f. g. }) \subseteq (\mathcal{S}_{f. g. }, \mathcal{S})$ by Theorem \ref{serre}. 
Let $M$ be in $(\mathcal{S}, \mathcal{S}_{f. g. })$. 
Then there exists a short exact sequence $0 \to L \to M \to M/L \to 0$ of $R$-modules where $L$ is in $\mathcal{S}$ such that $M/L$ is in $\mathcal{S}_{f. g. }$. 
It is easy to see that there exists a finitely generated $R$-submodule $K$ of $M$ such that $M=K+L$. 
Since $K \oplus L$ is in $(\mathcal{S}_{f. g. }, \mathcal{S})$ and $M$ is a homomorphic image of $K\oplus L$, 
$M$ is in $(\mathcal{S}_{f. g. }, \mathcal{S})$ by Proposition \ref{cusq}. 
\end{proof}

%%%%%%%%%%%%%%%%%%%%%%
%
%  Example
%
%%%%%%%%%%%%%%%%%%%%%%

\begin{example}
Let $R$ be a domain but not a field and let $Q$ be a field of fractions of $R$. 
We denote by $\mathcal{S}_{Tor}$ a subcategory consisting of torsion $R$-modules, 
that is 
\[\mathcal{S}_{Tor}=\{M\in R\text{-}\mathrm{Mod} \mid M \otimes_{R} Q=0 \}.\]

\vspace{2pt}
\noindent 
Then we shall see that one has   
\[ (\mathcal{S}_{Tor}, \mathcal{S}_{f. g. }) \subsetneqq (\mathcal{S}_{f. g. }, \mathcal{S}_{Tor})=\{M\in R\text{-}\mathrm{Mod} \mid \mathrm{dim}_{Q}\, M \otimes_{R} Q < \infty \}. \]

\vspace{2pt}
\noindent
Therefore, a subcategory $(\mathcal{S}_{f. g. }, \mathcal{S}_{Tor})$ is a Serre subcategory by Corollary \ref{FGSerre}, 
but a subcategory $(\mathcal{S}_{Tor}, \mathcal{S}_{f. g. })$ is not closed under extensions by Theorem \ref{serre}. 

\vspace{3pt}
First of all, 
we shall show that the above equality holds.  
We suppose that $M$ is in $(\mathcal{S}_{f. g. }, \mathcal{S}_{Tor})$. 
Then there exists a short exact sequence 
\[ 0 \to L \to M \to N \to 0 \]
of $R$-modules where $L$ is in $\mathcal{S}_{f. g. }$ and $N$ is in $\mathcal{S}_{Tor}$. 
We apply an exact functor $-\otimes_{R}Q$ to this sequence. 
Then we see that one has $M\otimes_{R}Q \cong L \otimes_{R}Q$ and this module is a finite dimensional $Q$-vector space. 

Conversely, 
let $M$ be an $R$-module with $\mathrm{dim}_{Q}\, M\otimes_{R}Q<\infty$. 
Then we can denote $M \otimes_{R} Q=\sum^n_{i=1} Q(m_{i} \otimes \displaystyle 1_{Q})$ with $m_{i} \in M$ and the unit element $1_{Q}$ of $Q$.  
We consider a short exact sequence
\[ 0 \to \sum^{n}_{i=1}Rm_{i} \to M \to M/\sum^{n}_{i=1}Rm_{i} \to 0\]
of $R$-modules. 
It is clear that $\sum^{n}_{i=1}Rm_{i}$ is in $\mathcal{S}_{f. g. }$ and $M/\sum^{n}_{i=1}Rm_{i}$ is in $\mathcal{S}_{Tor}$. 
So $M$ is in  $(\mathcal{S}_{f. g. }, \mathcal{S}_{Tor})$. 
Consequently, the above equality holds. 

\vspace{3pt}
Next, it is clear that $M \otimes_{R} Q$ has finite dimension as $Q$-vector space for an $R$-module $M$ of $(\mathcal{S}_{Tor}, \mathcal{S}_{f. g. })$. 
Thus, one has $(\mathcal{S}_{Tor}, \mathcal{S}_{f. g. }) \subseteq (\mathcal{S}_{f. g. }, \mathcal{S}_{Tor})$. 

\vspace{4pt}
Finally, we shall see that a field of fractions $Q$ of $R$ is in $(\mathcal{S}_{f. g. }, \mathcal{S}_{Tor})$ but not in $(\mathcal{S}_{Tor}, \mathcal{S}_{f. g. })$, 
so one has $(\mathcal{S}_{Tor}, \mathcal{S}_{f. g. }) \subsetneqq (\mathcal{S}_{f. g. }, \mathcal{S}_{Tor})$. 
\vspace{3pt}
Indeed, it follows from $\mathrm{dim}_{Q}\, Q \otimes _{R} Q=1$ that $Q$ is in $(\mathcal{S}_{f. g. }, \mathcal{S}_{Tor})$. % Z \subseteq M \subseteq Q に Q=R_{\NZD(R)}をtensor すればQ \subseteq M \otimes Q 
On the other hand, we assume that $Q$ is in $(\mathcal{S}_{Tor}, \mathcal{S}_{f. g. })$. 
Since $R$ is a domain, a torsion $R$-submodule of $Q$ is only the zero module. 
It means that $Q$ must be a finitely generated $R$-module. 
But, this is a contradiction. 
\end{example}

%%%%%%%%%%%%%%%%%%%%%%%%%%%%%%%%%%%%%%%%
%
%  Artin
%
%%%%%%%%%%%%%%%%%%%%%%%%%%%%%%%%%%%%%%%%
\vspace{5pt}
We note that a subcategory $\mathcal{S}_{Artin }$ consisting of Artinian $R$-modules is a Serre subcategory which is closed under injective hulls. (Also see \cite[Example 2.4]{AM-2008}.)  
Therefore  
we can see that a subcategory $(\mathcal{S}, \mathcal{S}_{Artin })$ is also Serre subcategory for a Serre subcategory of $R\text{-}\mathrm{Mod}$ by the following assertion. 

%%%%%%%%%%%%%%%%%%%%%%%%%%%%%%
%Serre of R-mod
%%%%%%%%%%%%%%%%%%%%%%%%%%%%%%%%

\begin{corollary}\label{CISerre}
Let $\mathcal{S}_{2}$ be a Serre subcategory of $R\text{-}\mathrm{Mod}$ which is closed under injective hulls. 
Then a subcategory $(\mathcal{S}_{1}, \mathcal{S}_{2})$ is a Serre subcategory for a Serre subcategory $\mathcal{S}_{1}$ of $R\text{-}\mathrm{Mod}$. 
\end{corollary}

\begin{proof}
By Theorem \ref{serre}, 
it is enough to show that one has $(\mathcal{S}_{2}, \mathcal{S}_{1}) \subseteq (\mathcal{S}_{1}, \mathcal{S}_{2})$. 

We assume that $M$ is in $(\mathcal{S}_{2}, \mathcal{S}_{1})$ and shall show that $M$ is in $(\mathcal{S}_{1}, \mathcal{S}_{2})$. 
Then there exists a short exact sequence 
\[0 \to S_{2} \to M  \to S_{1} \to 0 \]
of $R$-modules where $S_{1}$ is in $\mathcal{S}_{1}$ and $S_{2}$ is in $\mathcal{S}_{2}$.  
Since $\mathcal{S}_{2}$ is closed under injective hulls, we note that the injective hull $E_{R}(S_{2})$ of $S_{2}$ is also in $\mathcal{S}_{2}$. 
We consider a push out diagram

\[ \begin{CD}
0 @>>> S_{2} @>>> M @>>> S_{1} @>>> 0 \\
@.   @VVV @VVV    \parallel @.\\
0 @>>> E_{R}(S_{2})  @>>> X @>>> S_{1} @>>> 0 \\
\end{CD} \]

\vspace{3pt}
\noindent
of $R$-modules with exact rows and injective vertical maps. 
The second exact sequence splits, and we have an injective homomorphism $M \to S_{1} \oplus E_{R}(S_{2})$. 
Since there is a short exact sequence 
\[ 0 \to S_{1} \to S_{1} \oplus E_{R}(S_{2}) \to E_{R}(S_{2}) \to 0\]
of $R$-modules, the $R$-module $S_{1} \oplus E_{R}(S_{2})$ is in $(\mathcal{S}_{1}, \mathcal{S}_{2})$. 
Consequently, we see that $M$ is also in $(\mathcal{S}_{1}, \mathcal{S}_{2})$ by Proposition \ref{cusq}. 

The proof is completed.  
\end{proof}

%%%%%%%%%%%%%%%%%%%%%%%%%%%%%%%%
%closed under injective hull
%%%%%%%%%%%%%%%%%%%%%%%%%%%%%%%%

\begin{remark}
If subcategories $\mathcal{S}_{1}$ and $\mathcal{S}_{2}$ are Serre subcategories of $R\text{-}\mathrm{Mod}$ which are closed under injective hulls,  
then we can see that a subcategory $(\mathcal{S}_{1}$, $\mathcal{S}_{2})$ is Serre by Corollary \ref{CISerre} and closed under injective hulls as follows.

Let $M$ be in $(\mathcal{S}_{1}, \mathcal{S}_{2})$ and 
we shall prove that $E_{R}(M)$ is also in $(\mathcal{S}_{1}, \mathcal{S}_{2})$.
There exists a short exact sequence $0 \to L \to M \to N \to 0$ of $R$-modules where $L$ is in $\mathcal{S}_{1}$ and $N$ is in $\mathcal{S}_{2}$. 
We consider a commutative diagram 

$$  
\begin{CD}
0 @>>> L @>{\varphi}>>M  @>{\psi}>> N @>>> 0\\
@.  @VV{\sigma }V  @VV{\eta} V @VV{\tau }V \\
0 @>>>E_{R}(L)   @>>> E_{R}(L) \oplus E_{R}(N) @>>> E_{R}(N) @>>> 0 \\
\end{CD} $$

\vspace{3pt}
\noindent
of $R$-modules with exact rows and injective vertical maps. 
(For $m \in M$, 
we define $\eta (m)=(\mu (m), \tau \circ \psi (m))$ 
where $\mu :M \to E_{R}(L)$ is a homomorphism induced by the injectivity of $E_{R}(L)$ such that $\sigma=\mu \circ \varphi$.) 
Since $E_{R}(L) \oplus E_{R}(N)$ is in $(\mathcal{S}_{1}, \mathcal{S}_{2})$ and $E_{R}(M)$ is a direct summand of $E_{R}(L) \oplus E_{R}(N)$, 
$E_{R}(M)$ is in $(\mathcal{S}_{1}, \mathcal{S}_{2})$ by Proposition \ref{cusq}. 
\end{remark}

\vspace{7pt}
%%%%%%%%%%%%%%%%%%%%%%%%%%%%%
%%%%%%%%%%%%%%%%%%%%%%%%%%%%
%
%
%
%%%%%%%%%%%%%%%%%%%%%%%%%%%%%%%%%%%%%%%%%%%%%%%%%%%%
\section{On conditions $(C_{I})$ for a subcategory $(\mathcal{S}_{1}, \mathcal{S}_{2})$}

Let $I$ be an ideal of $R$. 
The following condition $(C_{I})$ was defined for Serre subcategories of $R\text{-}\mathrm{Mod}$ by M.\ Aghapouranahr and L.\ Melkersson in \cite{AM-2008}. 
In this section, 
we shall consider the condition $(C_{I})$ for subcategories of $R\text{-}\mathrm{Mod}$ and try to apply the notion of subcategories $(\mathcal{S}_{1}, \mathcal{S}_{2})$ to the local cohomology theory. 

\begin{definition}
Let $\mathcal{X}$ be a subcategory of $R\text{-}\mathrm{Mod}$ and $I$ be an ideal of $R$. 
We say that $\mathcal{X}$ satisfies the condition $(C_{I})$ if the following condition satisfied: 
\begin{center}
($C_{I}$) \hspace{5pt} If $M=\varGamma_{I}(M)$ and $(0 :_{M} I)$ is in $\mathcal{X}$, then $M$ is in $\mathcal{X}$. 
\end{center}
\end{definition}

\vspace{5pt}
If a subcategory $(\mathcal{S}_{1},\mathcal{S}_{2})$ is a Serre subcategory, 
we have already seen that one has $(\mathcal{S}_{2},\mathcal{S}_{1}) \subseteq (\mathcal{S}_{1},\mathcal{S}_{2})$ in Theorem \ref{serre}. 
First of all, we shall see that the condition $(C_{I})$ for a subcategory $(\mathcal{S}_{2}, \mathcal{S}_{1})$ induces the condition $(C_{I})$ for a Serre subcategory $(\mathcal{S}_{1}, \mathcal{S}_{2})$. 

\vspace{3pt} 
The following lemma is a well-known fact.

\begin{lemma}\label{ext}
Let $\mathcal{X}$ be a full subcategory of $R\text{-}\mathrm{Mod}$ that is closed under finite direct sums, submodules and quotient modules. 
Suppose that $N$ is finitely generated and that $M$ is in $\mathcal{X}$.  
Then $\mathrm{Ext}^i_{R}(N,M)$ is in $\mathcal{X}$ for all integer $i$.
\end{lemma}

\begin{theorem}\label{C_{I}}
Let $I$ be an ideal of $R$ and $\mathcal{S}_{1}, \mathcal{S}_{2}$ be Serre subcategories of $R\text{-}\mathrm{Mod}$. 
We suppose that $(\mathcal{S}_{1}, \mathcal{S}_{2})$ is a Serre subcategory of $R\text{-}\mathrm{Mod}$. 
If a subcategory $(\mathcal{S}_{2}, \mathcal{S}_{1})$ satisfies the condition $(C_{I})$, then $(\mathcal{S}_{1}, \mathcal{S}_{2})$ also satisfies the condition $(C_{I})$. 
\end{theorem}

\begin{proof}
Suppose that $M=\varGamma_{I}(M)$ and $(0 :_{M} I)$ is in $(\mathcal{S}_{1},\mathcal{S}_{2})$. 
Then we have to show that $M$ is in $(\mathcal{S}_{1}, \mathcal{S}_{2})$. 

Since $(0 :_{M} I)$ is in $(\mathcal{S}_{1}, \mathcal{S}_{2})$, 
there exists a short exact sequence 
\[ 0 \to L \to (0 :_{M} I) \to (0 :_{M} I)/L \to 0\]
of $R$-modules where $L$ is in $\mathcal{S}_{1}$ such that $(0 :_{M} I)/L$ is in $\mathcal{S}_{2}$. 
We consider a commutative diagram
$$  
\begin{CD}
0 @>>> L @>>>(0 :_{M} I)  @>>> (0:_{M} I)/L @>>> 0\\
@. \|  @.    @VVV @VVV \\
0 @>>>L   @>>> M @>>> M/L @>>> 0 \\
\end{CD} $$
of $R$-modules with exact rows. 
To prove our assertion, it is enough to show that
\[ M/L=\varGamma_{I}(M/L) \ \  \text{ and } \ \  (0 :_{M/L} I) \text{ is in } (\mathcal{S}_{2}, \mathcal{S}_{1}). \]
Indeed, it follows that $M/L$ is in $(\mathcal{S}_{2}, \mathcal{S}_{1})$ by the condition $(C_{I})$ for $(\mathcal{S}_{2}, \mathcal{S}_{1})$. 
Furthermore, since $(\mathcal{S}_{1}, \mathcal{S}_{2})$ is a Serre subcategory, 
one has $(\mathcal{S}_{2}, \mathcal{S}_{1}) \subseteq (\mathcal{S}_{1}, \mathcal{S}_{2})$ by Theorem \ref{serre}.  
Therefore, we see that $L$ is in $\mathcal{S}_{1}$ and $M/L$ is in $(\mathcal{S}_{1}, \mathcal{S}_{2})$, 
so it follows from the second exact sequence in a diagram that $M$ is in $(\mathcal{S}_{1}, \mathcal{S}_{2})$. 

The equality $M/L=\varGamma_{I}(M/L)$ is clear by $\mathrm{Supp}_{R}(M/L) \subseteq \mathrm{Supp}_{R}(M) \subseteq V(I)$.  
To see that $(0 :_{M/L} I) \text{ is in } (\mathcal{S}_{2}, \mathcal{S}_{1})$, 
we apply  a functor $\mathrm{Hom}_{R}(R/I, -)$ to  the short exact sequence 
\[ 0 \to L  \to M  \to  M/L \to  0.\]  
Then there exists an exact sequence
\[ 0 \to ( 0:_{L} I) \to (0 :_{M} I) \overset{\varphi}{\to} (0 :_{M/L} I) \to \mathrm{Ext}^{1}_{R}(R/I, L) \]
of $R$-modules.
It follows from $L \subseteq ( 0:_{M} I)$ that we have $( 0:_{L} I)=L$. 
Therefore, $\mathrm{Im} (\varphi)\cong (0:_{M} I)/L$ is in $\mathcal{S}_{2}$. 
Moreover, $\mathrm{Ext}^{1}_{R}(R/I, L)$ is in $S_{1}$ by Lemma \ref{ext}. 
Consequently, we see that $(0 :_{M/L} I)$ is in $(\mathcal{S}_{2}, \mathcal{S}_{1})$. 

The proof is completed. 
\end{proof}

%%%%%%%%%%%%%%%%%%%%%%%%%%%%%%%%%%%%%
%
% Application of local cohomology module. 
%
%%%%%%%%%%%%%%%%%%%%%%%%%%%%%%%%%%%%%
\vspace{5pt} 
Finally, 
we try to apply the notion of $(\mathcal{S}_{1}, \mathcal{S}_{2})$ to the local cohomology theory. 
It seems that we can rewrite several results concerned with Serre subcategories in local cohomology theory. 
However, we shall only rewrite \cite[7.1.6 Theorem]{BS} as demonstration here.

\begin{proposition}
Let $(R, \mathfrak{m})$ be a local ring and $I$ be an ideal of $R$.  
For Serre subcategories $\mathcal{S}_{1}$ and $\mathcal{S}_{2}$ of $R\text{-}\mathrm{Mod}$ which contain a non-zero module, 
we suppose that a subcategory $(\mathcal{S}_{1}, \mathcal{S}_{2})$ satisfies the condition $(C_{I})$. 
Then $H^{\mathrm{dim}\, M}_{I}(M)$ is in $(\mathcal{S}_{1}, \mathcal{S}_{2})$ for a finitely generated $R$-module $M$. 
\end{proposition}

\begin{proof}
We use induction on $n=\mathrm{dim}\, M$. 
If $n=0$, $M$ has finite length. 
It is clear that a Serre subcategory $\mathcal{S}$ contains all finite length $R$-modules if $\mathcal{S}$ contains a non-zero $R$-module. 
Therefore,  $\varGamma_{I}(M)$ is in $\mathcal{S}_{1} \cup \mathcal{S}_{2} \subseteq (\mathcal{S}_{1}, \mathcal{S}_{2})$. 

Now suppose that $n>0$ and we have established the result for finitely generated $R$-modules of dimension smaller than $n$. 
It is clear that $H^{i}_{I}(M) \cong H^{i}_{I}(M/\varGamma_{I}(M))$ for all $i>0$. 
Thus we may assume that $\varGamma_{I}(M)=0$. 
Then the ideal $I$ contains an $M$-regular element $x$. 
Therefore, there exists a short exact sequence 
\[ 0 \to M \overset{x}{\to} M \to M/xM \to 0, \] 
and this sequence induces an exact sequence
\[ H^{n-1}_{I}(M/xM)  \to H^{n}_{I}(M) \overset{x}{\to} H^{n}_{I}(M). \]
By the induction hypothesis, 
$H^{n-1}_{I}(M/xM)$ is in $(\mathcal{S}_{1}, \mathcal{S}_{2})$. 
Here, since $(\mathcal{S}_{1}, \mathcal{S}_{2})$ is closed under quotient modules, $(0:_{H^{n}_{I}(M)} x)$ is in $(\mathcal{S}_{1}, \mathcal{S}_{2})$. 
Furthermore, since $(0:_{H^{n}_{I}(M)} I) \subseteq (0:_{H^{n}_{I}(M)} x)$ and $(\mathcal{S}_{1}, \mathcal{S}_{2})$ is closed under submodules, 
$(0:_{H^{n}_{I}(M)} I)$ is in $(\mathcal{S}_{1}, \mathcal{S}_{2})$. 
It follows from the condition $(C_{I})$ for $(\mathcal{S}_{1}, \mathcal{S}_{2})$ that $H^{n}_{I}(M)$ is in $(\mathcal{S}_{1}, \mathcal{S}_{2})$. 

The proof is completed.  
\end{proof}

%%%%%%%%%%%%%%%%%%%%%%%%%
%
%
%  Acknowledgments
%
%
%%%%%%%%%%%%%%%%%%%%%%%%%

%\vspace{7pt}
%\section*{Acknowledgments}
%The author expresses  gratitude to the referees for their kind comments and valuable suggestions. 
%Thanks to referee's advice, 
%the proof for Corollary \ref{CISerre} became simpler than it in the earlier version of this paper.

\vspace{7pt}
%%%%%%%%%%%%%%%%%%%%%%%%%%%%%%%%%%%%%%%%%%%%%%%%%%%%%%%%%%%%%%%%%
%%%%%%%%%%%%%     Reference               %%%%%%%%%%%%%%%%%%%%%%%
%%%%%%%%%%%%%     論文→ デフォルト       %%%%%%%%%%%%%%%%%%%%%%%
%%%%%%%%%%%%%     本  →  イタリック      %%%%%%%%%%%%%%%%%%%%%%%
%%%%%%%%%%%%%%%%%%%%%%%%%%%%%%%%%%%%%%%%%%%%%%%%%%%%%%%%%%%%%%%%%


\begin{thebibliography}{99}
\addcontentsline{toc}{chapter}{Reference}%

%
\bibitem{AM-2008}
M. Aghapournahr and L. Melkersson, 
\textit{Local cohomology and Serre subcategories}, 
J. Algebra {\bf 320} (2008), 1275--1287.


%
\bibitem{BN-2008}
K. Bahmanpour and R. Naghipour, 
\textit{On the cofiniteness of local cohomology modules}, 
Proc. Amer. Math. Soc. {\bf 136} (2008), 2359--2363.

%
\bibitem{BS}
M. P. Brodmann and R. Y. Sharp, 
\textit{Local cohomology: an algebraic introduction with geometric applications},
Cambridge University Press, Cambridge (1998), xvi+416 pp.

%
\bibitem{G}
P. Gabriel, 
\textit{Des cat$\acute{e}$gories ab$\acute{e}$liennes}, 
Bull. Soc. Math. France \textbf{90} (1962), 323--448. 

%
\bibitem{Neeman-1992}
A. Neeman,  
\textit{The chromatic tower for $\mathcal{D}(R)$, With an appendix by Marcel B$\ddot{\text{o}}$kstedt}, 
Topology {\bf 31} (1992), 519--532.


%
\bibitem{Takahashi-2008}
R. Takahashi, 
\textit{Classifying subcategories of modules over a commutative noetherian ring}, 
J. Lond. Math. Soc. (2) {\bf 78} (2008), 767--782.


%
\bibitem{YY}
Y. Yoshino and T. Yoshizawa, 
\textit{Abstract local cohomology functors}, 
Math. J. Okayama Univ. {\bf 53} (2011), 129--154. 
\end{thebibliography}
\end{document}